\documentclass[10pt]{amsart}

\usepackage[latin1]{inputenc}   
\usepackage{amssymb,amsmath,amsthm,mathrsfs}
\usepackage{graphicx}
\usepackage[all]{xy}

\newcommand{\C}{\ensuremath{\mathbb{C}}}

\newcommand{\N}{\ensuremath{\mathbb{N}}}

\newcommand{\ra}{\ensuremath{\rightarrow}}

\newcommand{\D}{\ensuremath{\partial}}
\newcommand{\calo}{\ensuremath{\mathcal{O}}} 

\DeclareMathOperator{\Sing}{Sing}

\DeclareMathOperator{\Der}{Der}

\newcommand{\mc}[1]{\ensuremath{\mathcal{#1}}}   
\newcommand{\mf}[1]{\ensuremath{\mathfrak{#1}}}   

\newtheorem{thm}{Theorem}[section]
\newtheorem{cor}[thm]{Corollary}
\newtheorem{lem}[thm]{Lemma}
\newtheorem{question}[thm]{Question}
\newtheorem{proposition}[thm]{Proposition}

\theoremstyle{definition}
\newtheorem{defin}[thm]{Definition}
\newtheorem{rem}[thm]{Remark}
\newtheorem{exa}[thm]{Example}

\numberwithin{equation}{section}

\title{Towards transversality of singular varieties: splayed divisors}
\author{Eleonore Faber}
\address{
Dept. of Computer and Mathematical Sciences, University of Toronto at Scarborough, Toronto, Ont. M1A 1C4, Canada and Fakult\"at f\"ur Mathematik,
Universit\"at Wien, Austria;
}

\email{efaber@math.toronto.edu, eleonore.faber@univie.ac.at}
\thanks{
\noindent The author has been supported by a For Women in Science award 2011 of L'Or{\'e}al Austria, the Austrian commission for UNESCO and the Austrian Academy of Sciences and by
the Austrian Science Fund (FWF)
in frame of projects J3326 and P21461. This article contains work from the author's Ph.D. dissertation at the Universit\"at Wien. 
} 

\begin{document}

\begin{abstract}
We study a natural generalization of transversally intersecting smooth hypersurfaces in a complex manifold: hypersurfaces, whose components intersect in a transversal way but may be themselves
singular. Such
hypersurfaces will be called \emph{splayed\footnotemark} divisors. 
A splayed divisor is characterized by a property of its Jacobian ideal. This yields an effective test for splayedness. Two
further characterizations of a splayed divisor are shown, one reflecting the geometry of the intersection of its components, the other one using K.~Saito's logarithmic derivations. As an
application we prove that a union of smooth hypersurfaces has normal crossings if and only if it is a free divisor and has a radical Jacobian ideal. Further it is shown that the Hilbert--Samuel
polynomials of
splayed divisors satisfy the natural additivity property.
\end{abstract}

\maketitle


\section{Introduction}

\footnotetext[1]{The term ``splayed'' means ``spread out'' or ``made oblique''. So  being ``splayed'' for a hypersurface $D=D_1 \cup D_2$ should indicate that the components $D_1$ and $D_2$  are ``spread out'' the ambient space.}

Let $M_1, M_2$ be two submanifolds of a complex manifold $S$. Then $M_1$ and $M_2$ intersect \emph{transversally} at a point $p \in S$ if their respective tangent spaces add up to the tangent space of
$S$ at $p$, that is,
\begin{equation} \label{Equ:transversalsmooth}
T_pM_1 + T_pM_2=T_pS.
\end{equation}

Transversality is a fundamental concept in algebraic geometry as well as in differential geometry. 
However, in many applications (e.g., embedded resolution of singularities, Hodge structures) one needs a notion of transversality for more than two subspaces. This
leads to the notion of \emph{normal crossings}, which means that several smooth components cross in a transversal way. Another way to phrase this is that the union of several smooth subspaces is
isomorphic
to a union of coordinate subspaces. \\

In this article we study a natural generalization of the concept of transversal intersection of two hypersurfaces in complex manifolds allowing singular components. The geometric idea is that two
singular
hypersurfaces $D_1$ and $D_2$ in a complex manifold $S$ intersect transversally at a point $p$ if their ``tangent spaces'' fill out the whole space and the ideal of their intersection is reduced. The
notion of tangent space for singular hypersurfaces can be made precise by means of logarithmic derivations. In algebraic terms this means: one
can find local coordinates at $p$ such that the
defining equations of the $D_i$ can be chosen in separated variables. We call a union of such transversally
intersecting hypersurfaces  a \emph{splayed  divisor}.  This concept has already appeared in different contexts under different names, for example, J.~Damon \cite{Damon96} called germs of the form $(V_1 \times \C^m) \cup (\C^n \times V_2) \subseteq \C^{n+m}$, where $V_1 \subseteq \C^n$ and $V_2 \subseteq \C^m$, a \emph{product-union}, also see Remark \ref{Rem:Damon}. For an example of a splayed divisor in a threedimensional $S$, see Figure \ref{fig:splayedcusp}.  \\
\begin{figure}[!h]
\begin{tabular}{c@{\hspace{1.5cm}}c}
\includegraphics[width=0.39 \textwidth]{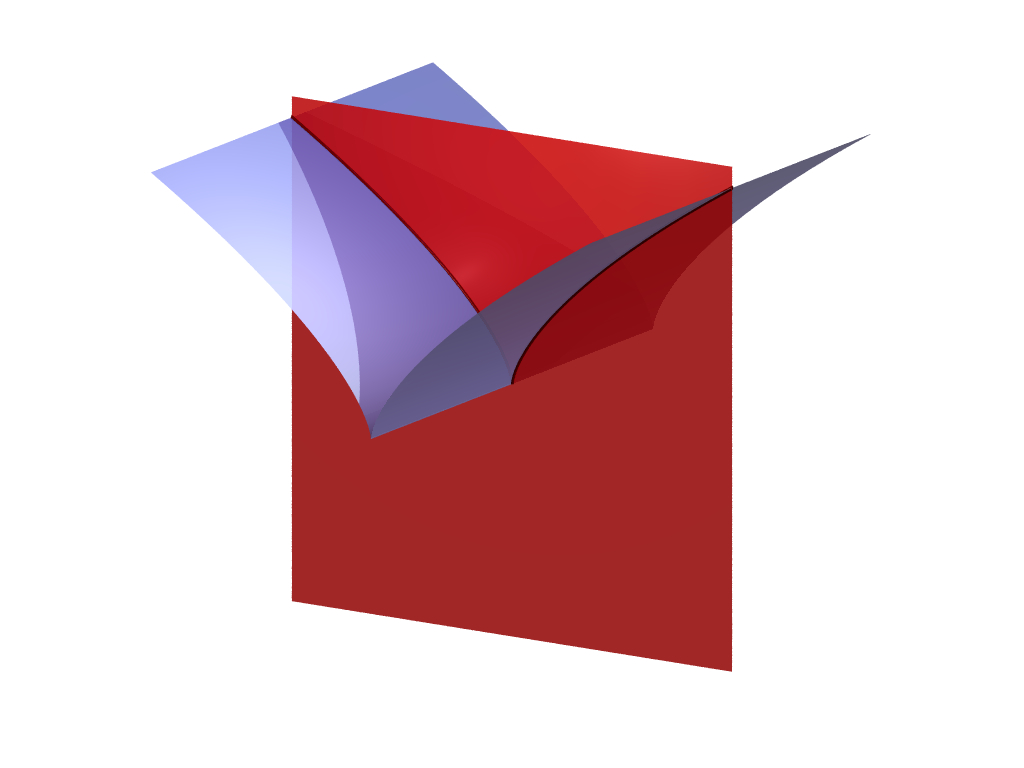}& 
\includegraphics[width=0.39 \textwidth]{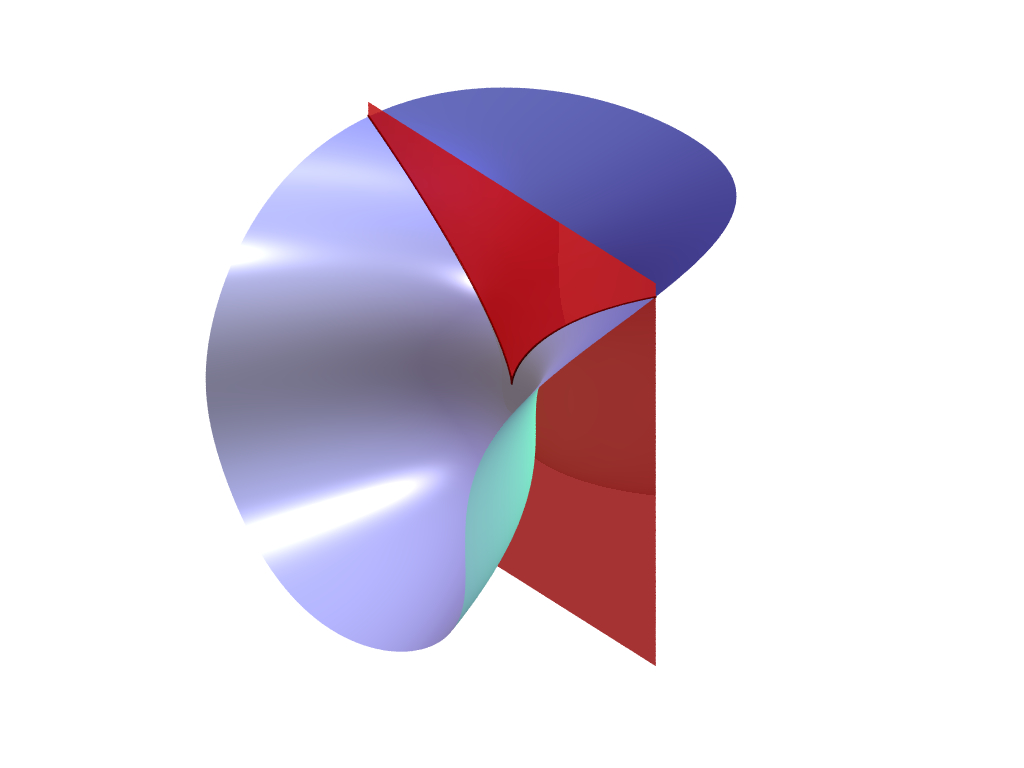}
\end{tabular}
\caption{ \label{fig:splayedcusp}  $D=\{x(y^2-z^3)=0\}$ (left) is splayed and  $D'=\{x(x+y^2-z^3)=0\}$ (right) is not splayed.}
\end{figure}

In the first part of this work three different characterizations of splayed divisors are shown. Suppose that a divisor $D \subseteq S$ is given by $(D,p)=(D_1,p) \cup (D_2,p)=\{gh=0\}$. First we
consider the Jacobian ideal $J_{gh}$ (the ideal generated by the partial derivatives of $gh$) of $D$ at $p$. It is clear that for a splayed $D$, the Jacobian ideal satisfies 
$$ J_{gh}=gJ_h+hJ_g,$$
when the defining equations $g$ and $h$ are chosen in separated variables.
We show (Proposition \ref{Thm:Leibnizprod}) that this property already characterizes splayed divisors. \\
Further, for a splayed divisor $D$, the ideal $((gh)+J_{gh})$ defining its singular locus can
be written as the intersection of the ideals defining the singular loci of the $D_i$ and of the ideal defining the
intersection $D_1 \cap D_2$:
$$((gh)+J_{gh})=(g,h) \cap ((g)+J_g) \cap ((h)+J_h).$$ 
In Proposition \ref{Thm:splayedgeom} the sufficiency of this condition is proved. \\
Finally, equation (\ref{Equ:transversalsmooth}) can be directly translated to singular subspaces via
K.~Saito's logarithmic derivations \cite{Saito80},
namely, $D$ is splayed at $p$ if and only if 
$$ \Der_{S,p}(\log D_1) + \Der_{S,p}(\log D_2)=\Der_{S,p}.$$
This is shown in Proposition \ref{Thm:splayedder}. \\

In the second part, two applications are considered. First
the relationship between splayed divisors and free divisors is studied. Free divisors are a generalization of normal crossing divisors and appear frequently in different areas of
mathematics, for example, in deformation theory as discriminants or in combinatorics as free hyperplane arrangements, see \cite{Aleksandrov86, OrlikTerao92, Buchweitz06, MondSchulze10, BuchweitzMond,
Saito81} for more examples.
Then we give a partial answer to a question of H.~Hauser about the characterization of normal crossing divisors by their Jacobian ideals: it is shown that if $D=\bigcup_{i=1}^n D_i$ is locally the
union of smooth
irreducible components then $D$ has normal crossings if and only if $D$ is locally free and its Jacobian ideal is radical. \\
The computation of singularity invariants of splayed divisors is very interesting, in particular one should be able to deduce them from singularity invariants of their splayed components. We start
with computing the Hilbert--Samuel polynomial $\chi_{D,p}$ of a splayed divisor $(D,p)=(D_1,p) \cup (D_2,p)$ and find that it satisfies an additivity relation: from the exact sequence
$$ 0 \rightarrow \mc{O}_{D,p} \rightarrow \mc{O}_{D_1,p} \oplus \mc{O}_{D_2,p}
\rightarrow \mc{O}_{D_1 \cap D_2,p} \ra 0$$
we deduce 
$$\chi_{D,p}(t)+\chi_{D_1 \cap D_2,p}(t)=\chi_{D_1,p}(t)+\chi_{D_2,p}(t).$$ 
This additivity relation does not characterize
splayed divisors, an example therefore is given at the end.

\section{Setting and triviality lemma} \label{Sec:preliminaries}

We work in the complex analytic category. The main objects of our study are divisors (=hypersurfaces) in complex manifolds. We write $(S,D)$ for a fixed divisor $D$ in an $n$-dimensional complex
manifold $S$.
Denote by $p$ a point in $S$ and by $x=(x_1, \ldots, x_n)$ the complex coordinates of $S$ around $p$. The divisor $(D,p)$ will then be defined locally by an equation $\{
h_p(x_1, \ldots, x_n)=0 \}$ where $h_p \in \mc{O}_{S,p} \cong \C\{x_1, \ldots, x_n\}$ (if the context is clear we leave the subscript $_p$). Note that we will always assume that $h$ is reduced! To
simplify the notation regarding splayed divisors, sometimes the coordinates are denoted by $(x_1, \ldots, x_k,y_{k+1}, \ldots, y_n)$ and $h(x,y) \in \C\{x,y\}$.
The divisor $D$ has \emph{normal crossings} at a point $p$ if one can find complex coordinates $(x_1, \ldots, x_n)$ at $p$ such that the defining equation $h$ of $D$ is $h=x_1 \cdots x_k$ for $k
\leq n$. We also say that $(D,p)$ is a \emph{normal crossing divisor}. 
The \emph{Jacobian ideal} of $h$ is denoted by $J_{h}=(\D_{x_1} h, \ldots, \D_{x_n}h) \subseteq \mc{O}_{S,p}$. 
The image of $J_h$ under the  canonical epimorphism that sends $\mc{O}_{S,p}$ to $\mc{O}_{D,p}=\mc{O}_{S,p}/(h)$ is denoted by $\widetilde{J_h}$. 
The associated analytic coherent ideal sheaves are
denoted by $J \subseteq \mc{O}_{S}$ and $\widetilde J$ in $\mc{O}_{D}$.   
The singular locus $\Sing D$ of $D$ is defined by the ideal sheaf $\widetilde J \subseteq \mc{O}_{D}$. The singular locus $(\Sing D,p)$ is defined by the local ring
$\mc{O}_{\Sing D, p}=\mc{O}_{S,p}/((h) + J_h)$. \\

The next lemma yields an ideal-theoretic characterization of Cartesian product structure. It is used frequently and can be found in various
different formulations in the literature (for example in \cite{dJP,Saito80,GH,CNM96}). 

\begin{lem}[Triviality lemma] \label{Lem:saitotriv}
Let $(S,p)$ locally be isomorphic to $(\C^{n+m},0)$ and denote
$\mc{O}_{S,p}=\C\{ x_1, \ldots, x_n, y_1, \ldots, y_m\}$ and let $h(x_1, \ldots, x_n,$ $y_1, \ldots , y_m)$ be
an element of $\mc{O}_{S,p}$. Then the following are equivalent: \\
(a) The ideal $(\D_{y_1}h, \ldots, \D_{y_m}h)$ is contained in the ideal $(h,
\D_{x_1}h, \ldots, \D_{x_n}h)$. \\
(b) There exists a local biholomorphic map $\varphi: (\C^{n+m},0) \rightarrow
(\C^{n+m},0)$ and a holomorphic $v(x,y) \in \mc{O}_{S,p}^*$ such that
$$\varphi(x,y)=(\varphi_1(x,y), \ldots, \varphi_n(x,y),y_1, \ldots, y_m),$$
$\varphi(x,0)=(x,0)$, $v(x,0) \equiv 1$ and $h \circ \varphi (x,y)=v(x,y)h(x,
0)$.  \\
This means that $D=\{h(x,y)=0\}$ is locally at $p$ isomorphic to some $(D' \times \C^m, (0,0))$ where $D'$ is locally contained in $\C^n$. \\
\end{lem}

\begin{rem}
 This lemma is called ``triviality lemma'' because a strong\-er form of it characterizes local analytic triviality: if in $(a)$ all partial derivatives $\D_{y_i}h$ are even contained in
$(x_1, \ldots, x_n, y_1, \ldots, y_m)(\D_{x_1}h, \ldots, \D_{x_n}h)$, then one can show that $(D,p) \cong (D_0 \times \C^m,(p',0))$ where $D_0=\{ h(x,0)=0\}$ is the ``fibre'' at the origin.
\end{rem}

\begin{proof}
See for example \cite{Saito80}.
\end{proof}

The objects of our studies are divisors that are unions of Cartesian products. We call them splayed because they fill out the whole space. 

\begin{defin} \label{Def:produkt}
Let $D$ be a divisor in a complex manifold $S$, $\dim S=n$. The
divisor $D$ is called \emph{splayed} at a point $p \in S$ (or $(D,p)$ is \emph{splayed}) if one can find coordinates $(x_1, \ldots, x_n)$ at $p$ such that $(D,p)=(D_1,p) \cup (D_2,p)$ is defined by
$$h(x)=h_1(x_1, \ldots, x_k) h_2(x_{k+1}, \ldots, x_n),$$
$1\leq k \leq n-1$, where $h_i$ is the defining reduced equation of $D_i$. Note that the $h_i$ are not necessarily irreducible. The
\emph{splayed components} $D_1$ and $D_2$ are not unique. Splayed means that $D$ is the union of two products: since $h_1$ is independent of $x_{k+1}, \ldots, x_n$, the divisor $D_1$ is locally at $p$
isomorphic to a product $(D'_1,0) \times (\C^{n-k},0)$, where $(D'_1,0) \subseteq (\C^k,0)$ (and similar for $D_2$). 
\end{defin}

\begin{exa} (1) Let $(D,0)$ be the divisor in $(\C^2,0)$ defined by $h_1 h_2=x (y-x^2)$. Since $D$ has normal crossings at the origin, $D$ is splayed. \\
(2) The divisor $D=\{ (x-y^2)zw=0\}$ is splayed in $(\C^4,0)$ but its splayed components are not unique, for example $h_1=x-y^2$ and $h_2=zw$ or $h_1=(x-y^2)w$ and $h_2=z$. \\
(3) The divisor $D=\{ (xz(x+z-y^2)=0\}$ is not splayed in $(\C^3,0)$.
\end{exa}

\begin{rem}  \label{Rem:Damon}
The concept of splayed divisors was also introduced by J. Damon in the context of free divisors  under the name product union, see \cite{Damon96}. In the theory of hyperplane arrangements one
studies the analogue concept, under various names: e.g., in \cite{OrlikTerao92} splayed arrangements are called reducible and in \cite{Budur10} they are referred to as decomposable.
\end{rem}

\section{Criteria for splayed divisors} \label{Sec:criteria}

Let $(D,p)$ be a divisor in $(S,p)$. If the decomposition into irreducible components of $D$ is known, then there are effective methods to test whether $D$ has normal crossings at $p$, see 
\cite{Bodnar04}. The idea here is to find linearly independent tangent vectors, that is, to linearize the problem and thus reduce it to linear algebra. The algorithm of Bodn{\'a}r does
not work in order to test whether $(D,p)$ is splayed, since the irreducible components may be singular themselves at $p$.  \\
In Proposition \ref{Thm:Leibnizprod} it is shown that $D=D_1 \cup D_2=\{
gh=0\}$ is splayed at $p$ if and only if, up to multiplying $g, h$ with units, its
Jacobian ideal 
$J_{gh}$ of $(D,p)$ is equal to $gJ_h+hJ_g$. This property can be tested in computer algebra systems like {\sc Singular}. The second characterization is more geometrical. In Proposition \ref{Thm:splayedgeom} we show that for a splayed
divisor $\Sing D$ is the scheme-theoretic union of the intersection of $D_1$ and $D_2$, and of the singular loci $\Sing D_1$ and $\Sing D_2$.
The third criterion (Proposition \ref{Thm:splayedder}) is nearer to Bodn{\'a}r's normal crossings test. Therefore we use the notion of
logarithmic
derivations and differential forms (in the sense of K.~Saito \cite{Saito80}), which will be a main ingredient in the section about freeness of splayed divisors.

\subsection{Jacobian ideal characterization 1 - algebra} \label{Sub:leibniz}

Here a characterization of splayedness by Jacobian Ideals is shown, which can be easily checked in concrete examples.

\begin{defin}
Let $D_1=\{g=0\}$, $D_2=\{h=0\}$ and $D=\{gh=0\}$ at $p$ be defined as above. We say that $J_{gh}$ satisfies the \emph{Leibniz property} if
$$J_{gh}=gJ_h + hJ_g.$$
\end{defin}

If $D$ is splayed then it is clear that its Jacobian ideal satisfies the Leibniz property. To establish the other implication, first an intermediate ideal-theoretic characterization of splayedness is
proven (Lemma
\ref{Lem:produktallg}). 

\begin{lem}  \label{Lem:produktallg}
Let $\dim S=n$ and at a point $p=(x_1, \ldots, x_n)$ denote by $\calo=\mc{O}_{S,p}=\C\{x_1, \ldots, x_n\}$ the local ring at $p$. Let $D_1=\{g(x)=0\}$, $D_2=\{h(x)=0\}$ and
$D=\{ gh(x)=0\}$ be divisors, where we assume that $g,h \in \mc{O}$ are reduced and have no common factors. Then $(D,p)=(D_1,p) \cup D_2,p)$ is splayed if and only if
\begin{equation}   \label{Equ:splayeddurchschnitt}
(g) \cap ((gh)+J_{gh})= g((h) + J_h).
\end{equation}
\end{lem}

\begin{proof}
First note that the inclusion $(g) \cap ((gh)+J_{gh})\subseteq  g((h) + J_h)$ always holds since $g$ and $h$ have no common factors. \\
If $D$ is splayed, we can suppose without loss of generality that $D_1=\{g(x,0)=0\}$ and $D_2=\{h(0,y)=0\}$ where  $(x,y)=(x_1, \ldots, x_k,
y_{k+1},\ldots, y_n)$. Then it is easy to see, via the Leibniz property of $J_{gh}$, that $g((h)+J_h)$ is contained in $(g) \cap ((gh)+J_{gh})$.
For the other direction, suppose that  (\ref{Equ:splayeddurchschnitt}) holds. 
In this case it is easy to see that (\ref{Equ:splayeddurchschnitt}) is preserved under local isomorphisms of $(S,p)$.
One may assume without loss of generality that $h(x_1, \ldots, y_n)=h(0, \ldots,0,$ $y_{k+1}, \ldots, y_n)$ and that
\begin{equation} \label{Equ:nichtinm}
 \D_{y_i}h \not \in (h,\D_{y_{k+1}}h, \ldots, \widehat{ \D_{y_{i}}h}, \ldots,
\D_{y_n}h)
\end{equation}
for all $i \in \{ k+1, \ldots, n\}$. (If not so, suppose that e.g. $\D_{x_1}h$ is contained in $(h,\D_{x_{2}}h, \ldots, \D_{x_n}h)$. Then by the triviality lemma there exists a locally biholomorphic
map
$\varphi: (S,p) \rightarrow (S,p)$  such that 
$h \circ \varphi(x)=v(x) h(0, x_{2}, \ldots, x_n)$, with $v \in \mc{O}^*$. Then
 set $\widetilde h:= h(0,x_{2}, \ldots, x_n)$  and $\widetilde g:= g \circ \varphi$. The divisor defined by $\widetilde g \cdot \widetilde h$ is clearly isomorphic to $D$, and similarly $D_1$ and
$D_2$ are defined by $\widetilde g$ and $\widetilde h$.) \\
Now let $h(y)=h(0, \ldots, 0, y_{k+1}, \ldots, y_n)$ such that (\ref{Equ:nichtinm}) holds. From (\ref{Equ:splayeddurchschnitt}) it follows that $g (\D_{y_i}h)$ is contained in $((gh)+J_{gh})$ and by
definition $g \D_{y_i}h + h \D_{y_i}g$ is also contained in this ideal. Hence $h(\D_{y_i}g) \in ((gh)+J_{gh})$ for all $k+1 \leq i \leq n$. 
Thus we can write
$$h (\D_{y_i}g)= a_i gh + \sum_j a_{ij}(g \D_{y_j}h + h \D_{y_j}g)+\sum_j b_{ij}h\D_{x_j}g$$
where the summation indices run over the valid range. This equation can be rearranged as
\begin{equation} \label{Equ:splayed2}
 h(\D_{y_i}g -  a_i g - \sum_j a_{ij} \D_{y_j}g - \sum_j b_{ij}\D_{x_j}g)  =g \sum_j a_{ij} (\D_{y_j}h).
\end{equation}
Since $g$ and $h$ have by assumption no common factors, we conclude from (\ref{Equ:splayed2}) that $\D_{y_i}g -  a_i g - \sum_j a_{ij} \D_{y_j}g - \sum_j b_{ij}\D_{x_j}g \in (g)$ and that $\sum_j
a_{ij} (\D_{y_j}h)=c_i h$ for any $ k+1 \leq i \leq n$ and some $c_i \in \mathcal{O}$. Then (\ref{Equ:nichtinm}) implies that any $a_{ij} \in \mf{m}$: if this were not the case, that is, if 
some $a_{im} \in \mc{O}^*$, then via the equation $\sum_j a_{ij} (\D_{y_j}h)=c_i h$ one can express $\D_{y_m}h$ in terms of $h$ and the other $\D_{y_j}h$, which is a contradiction to
(\ref{Equ:nichtinm}). Hence it follows that 
$$\D_{y_i}g \in (g, \D_{x_1}g, \ldots, \D_{x_k}g)+\mf{m}(\D_{y_{k+1}}g, \ldots, \D_{y_n}g)$$
for any $k+1 \leq i \leq n$. By an application of Nakayama's lemma we obtain
$$\D_{y_i}g \in (g, \D_{x_1}g, \ldots, \D_{x_k}g) \text{ for all }i=k+1, \ldots,
n.$$
By the triviality lemma  there exists a locally biholomorphic $\psi: (S,p) \ra (S,p)$ with $\psi(x,y)=(\psi_1(x,y), \ldots, \psi_k(x,y),
y_{k+1}, \ldots, y_n)$ 
 such that $g \circ \psi$ is equal to $vg(x_1, \ldots, x_{k},0)$ for a unit $v \in \calo$. Set
$\tilde g:=v^{-1}(g \circ \psi)$ and $\tilde h:=h \circ
\psi=h=h(0, \ldots, 0, y_{k+1}, \ldots, y_n)$. By
construction $\tilde g \tilde h$ defines a splayed divisor that is isomorphic to
$D$. 
\end{proof}

\begin{proposition}  \label{Thm:Leibnizprod}
Let $(S,D)$ be a complex manifold $S$, $\dim S=n$, together with a divisor $D \subseteq S$ that is locally at a point defined by $\{gh=0\}$, where $g$ and $h$ are reduced elements
of $\mc{O}_{S,p}$ that are not necessarily irreducible but have no common
factor. Then $D=\{ g=0\} \cup \{h=0\}$ is  splayed at $p$ if $J_{gh}$ satisfies the Leibniz property
$$ J_{gh}=gJ_h + h J_g.$$
Conversely, if $D=\{ gh = 0\}$ is splayed and $g$ and $h$ are chosen in different variables, then $J_{gh}$ satisfies the Leibniz property. This means that up to possible multiplication of $g$ and $h$ with units, $J_{gh}$ satisfies the Leibniz property.
\end{proposition}

\begin{proof}
As already remarked, if $D$ is splayed then we can choose the defining equations $g$ and $h$ in separated variables and it is clear that $J_{gh}=gJ_h+hJ_g$. Conversely, it is enough to show the equality $(g) \cap ((gh)+J_{gh})=g((h)+J_h)$ from lemma \ref{Lem:produktallg}. Since
$(g) \cap ((gh)+J_{gh})$ is always contained in $g((h)+J_h)$ it remains to check the other inclusion: take a $\beta \in g((h)+J_h)$. Then a  straightforward calculation shows that
$\beta$ is also contained in $(g) \cap ((gh)+J_{gh})$.
\end{proof}

As pointed out by P.~Aluffi, $D=\{ gh =0\}$ is splayed if and only if 
\begin{equation} \label{Equ:leibnizmod}
(gh)+J_{gh}=(gh)+gJ_h+hJ_g,
\end{equation}
regardless of multiplying $g$ and $h$ by units. Equation (\ref{Equ:leibnizmod}) can also easily be shown with lemma \ref{Lem:produktallg}. For a different approach to this characterization of splayedness, see \cite{AluffiFaber12}.

\begin{exa} Let $D$ be the divisor in $\C^3$ given at a point $p$ by $x(x+y^2-z^3)$. Then $D$ is the union of two smooth components $D_1=\{h=x+y^2-z^3=0\}$ and $H=\{g=x=0\}$. The ideal
$(gh)+J_{gh}=J_{gh}=(2x+y^2-z^3, xy,xz^2)$ is strictly contained in $(gh)+gJ_h+hJ_g=gJ_h+hJ_g=(x,y^2-z^3)$.  Thus $D$ is not a splayed divisor (see Figure \ref{fig:splayedcusp}). 
\end{exa}

\subsection{Jacobian ideal characterization 2 - geometry} \label{Sub:geometry}
Now we characterize a splayed divisor $D \subseteq S$, locally at a point $p$ given by a $gh \in \mc{O}_{S,p}$, by $\mc{O}_{\Sing D,p}=\mc{O}_{S,p}/((gh)+J_{gh})$. This characterization
 reflects the geometry of $(D,p)$, namely the two splayed components meeting transversally.

\begin{proposition}  \label{Thm:splayedgeom}
The divisor $D=D_1 \cup D_2$, defined at $p$ as above, is splayed if and only if 
\begin{equation} \label{Equ:splayedgeom}
((gh)+J_{gh})=(g,h) \cap ((g)+J_g) \cap ((h)+J_h).
\end{equation}
\end{proposition}

\begin{proof}
Recall here from lemma \ref{Lem:produktallg} that a divisor $\{gh=0\}$ is splayed if and only if (\ref{Equ:splayeddurchschnitt}) holds.
First suppose that (\ref{Equ:splayedgeom}) holds. A straightforward calculation shows (\ref{Equ:splayeddurchschnitt}). \\
For the other implication we use Grauert's division theorem (for the notation and statement see \cite[Theorem 7.1.9]{dJP}): Suppose that $D=\{gh=0\}$ is splayed. Then without loss of generality $g(x,y)=g(x)$ and
$h(x,y)=h(y)$ in $\C\{x_1, \ldots, x_k, y_{k+1},$ $
\ldots, y_n\}$. Clearly $((gh)+J_{gh})
\subseteq (g,h) \cap ((g)+J_g) \cap ((h)+J_h)$. So let $\alpha$ be an element of the right-hand side, that is, $\alpha=ag+bh=cg+\sum_{i=1}^k a_i \D_{x_i}g$ for some $a,b,c, a_i \in \mc{O}$.
Then also $\alpha
-ag=bh=(c-a)g+\sum_{i=1}^ka_i \D_{x_i}g$ is contained in $((g)+J_g)$. By Grauert's division theorem there exist some $\tilde a, \tilde a_i, r, \tilde r_i$ such that for all
$i=1, \ldots, k$ one has $c-a=\tilde a h +r$ and $a_i=\tilde a_i h +r_i$  and the leading monomial $L(h)$ does not divide any monomial of the unique remainders $r, r_i$. Then write
$$(b-\tilde a g - \sum_{i=1}^k \tilde a_i (\D_{x_i}g))h=rg +\sum_{i=1}^k r_i (\D_{x_i}g).$$
Since $h$ only depends on $y$ and $g$ only on $x$, it follows that $L(h)$ does also not divide any of the monomials of the right-hand side of the equation. But this is only possible 
if $(b-\tilde a g - \sum_{i=1}^k \tilde a_i (\D_{x_i}g))h=0$. 
It follows that $b$ is contained in $((g)+J_g)$. Interchanging the roles of $g$ and $h$
yields that $a \in ((h)+J_h)$ and thus $\alpha \in (gh, gJ_h +hJ_g)$. The Leibniz property of $J_{gh}$ implies that $\alpha \in ((gh)+ J_{gh})$.
\end{proof}

\begin{rem}
Here it can be seen that for splayed divisors the Jacobian ideal is the intersection of the two ideals defining the singular loci of the splayed components $D_1$ and $D_2$ plus the intersection of
$D_1$ and
$D_2$. For two smooth divisors $D_1$ and $D_2$ this means that $D=D_1 \cup D_2$ is a splayed divisor if and only if the scheme-theoretical intersection $D_1\cap D_2$ is smooth, which
is in turn equivalent to saying that $D_1$ and $D_2$ intersect transversally, (cf.~\cite{LiLi,FaberHauser10}). Note that for varieties of arbitrary codimension the condition that the
scheme-theoretical intersection is smooth, leads to the formulation of \emph{clean intersection}, see \cite{LiLi}.
\end{rem}

\begin{exa}   \label{Ex:splayedcusp}
Let $D \subseteq \C^3$ be given at the origin by $gh=x(x+y^2+z^3)=0$. Then $D$ is not splayed at the origin. Here the intersection of the two components is given by the ideal $(g,h)=(x,y^2+z^3)$. Also
consider the divisor $D' \subseteq \C^3$ that is given by $g'h'=x(y^2+z^3)$. Clearly $D'$ is splayed at the origin and the intersection of the two components is given by the ideal
$(g',h')=(x,y^2+z^3)$, see Figure \ref{fig:splayedcusp}.
\end{exa}

\subsection{Logarithmic derivation characterization - tangency} \label{Sub:logder}
As already mentioned in the introduction,  two submanifolds of a manifold intersect transversally at a point if their respective tangent spaces at that point together generate the tangent space of the
ambient manifold. In this section we will translate this definition for splayed divisors with the help of logarithmic derivations (Proposition \ref{Thm:splayedder}). Logarithmic derivations
were introduced by K.~Saito in \cite{Saito80} and are a useful tool in studying tangent behaviour for singular varieties. They lead to the definition of free divisors, which will be studied in more
detail in section \ref{Sub:freediv}.

\subsubsection{A very brief recap of Saito's theory of free divisors (\cite{Saito80})}
Let $D$ be a divisor in $S$ defined at $p$ by $D=\{h=0\}$. A \emph{logarithmic vector field} (or \emph{logarithmic derivation}) (along $D$) is a holomorphic vector field  on $S$, that is, an element
of $\Der_{S}$, satisfying
one of the two equivalent
conditions: \\
(i) For any smooth point $p$ of $D$, the vector $\delta(p)$ of $p$ is tangent to $D$, \\
(ii) For any point $p$, where $(D,p)$ is given by $h=0$, the germ $\delta(h)$ is contained in the ideal
$(h)$ of  $\mc{O}_{S,p}$. The module of germs of logarithmic derivations (along $D$) at $p$ is denoted by 
\[  \Der_{S,p}(\log D)=\{  \delta: \delta \in \Der_{S,p} \text{ such that }\delta h \in (h)  \}, \]
The $\Der_{S,p}(\log D)$ are the stalks at points $p$ of the sheaf $\Der_S(\log D)$ of $\mc{O}_{S}$-modules. 
Similarly we define logarithmic differential forms: a meromorphic $q$-form $\omega$ is logarithmic (along
$D$) at a point $p$ if $\omega h$ and $hd\omega$ are holomorphic in an
open neighbourhood around $p$. We write
\[\Omega^q_{S,p}(\log D)= \{ \omega: \omega \text{ germ of a logarithmic $q$-form at $p$} \},\] 
One can show that $\Der_{S,p}(\log D)$ and $\Omega^1_{S,p}(\log D)$ are reflexive $\mc{O}_{S,p}$-modules dual to each other (see \cite{Saito80}). One says that $(D,p)$ is \emph{free} or that $D$ is
free at $p$ if
$\Der_{S,p}(\log D)$ resp. $\Omega_{S,p}^1((\log D)$ is a free $\calo_{S,p}$-module. In section \ref{Sub:freediv} free divisors are discussed in more detail and the following theorem is used, which
makes it possible to test whether $D$ is free (cf.~\cite[Thm.1.8]{Saito80}): 

\begin{thm}[Saito's criterion] \label{Thm:Saito}
Let $(S,D)$, $p$ and $h$ be as defined in the introduction. The $\mc{O}_{S,p}$-module $\Der_{S,p}(\log D)$ is free if and only if there exist $n$ vector fields $\delta_i=\sum_{j=1}^n
a_{ij}(x) \D_{x_j}$
in $\Der_{S,p}(\log D)$, $i=1, \ldots, n$, such that $\det(a_{ij}(x))$ is equal to $h$ up to an invertible factor. Moreover, then the vector fields $\delta_1, \ldots, \delta_n$ form a basis for
$\Der_{S,p}(\log
D)$. \\
\end{thm}

\begin{rem}
The module $\Der_{S,p}(\log D)$ naturally carries the structure of a Lie algebra. This was considered under the name \emph{tangent algebra} in \cite{HM93}, where also many properties of
algebraic and analytic varieties were characterized in terms of tangent algebras. The Lie algebra structure of logarithmic derivations is also considered in \cite{GrangerSchulze09,Sekiguchi08}.
\end{rem}

\subsubsection{Logarithmic derivations and splayed divisors} 

\begin{proposition}  \label{Thm:splayedder}
Let $(D,p)=(D_1, p) \cup (D_2,p) \subseteq (S,p)\cong (\C^n,0)$ be a divisor such that $(D_1,p)=\{ g=0\}$ and $(D_2,p)=\{h=0\}$ for some $g,h$ not necessarily irreducible but with no common factor.
Then $D$ is
splayed at $p$ if and only if 
\begin{equation}\Der_{S,p}(\log D_1)+ \Der_{S,p}(\log D_2)=\Der_{S,p}.
 \label{Equ:splayednew}
\end{equation}
\end{proposition}

\begin{proof}
As above, we denote by $(x,y):=(x_1, \ldots, x_k, y_{k+1}, \ldots, y_n)$ the coordinates at $p$ and by $\calo=\C\{x,y\}$ the corresponding local ring. If $D$ is splayed then without loss of generality one has $g=g(x,0)$
and $\Der_{S,p}(\log D_1)$ is generated by some $\delta_1, \ldots, \delta_m$, $m \geq k$, $\D_{y_{k+1}},\ldots, \D_{y_n}$. Similarly $h=h(0,y)$ and $\Der_{S,p}(\log D_2)$ is generated by some
$\D_{x_1}, \ldots, \D_{x_k}, \varepsilon_{1}, \ldots, \varepsilon_l$ for some $l \geq n-k$. Then clearly we have
$$ \Der_{S,p}(\log D_1)+\Der_{S,p}(\log D_2)  ={}_\calo\!\langle \D_{x_1}, \ldots, \D_{x_k}, \D_{y_{k+1}}, \ldots, \D_{y_n}\rangle = \Der_{S,p}.$$
For the other implication we may assume that $g(x,y)=g(x,0)$ and 
$$\D_{x_i}g \not \in (\D_{x_1}g, \ldots, \widehat{\D_{x_i}g}, \ldots, \D_{x_k}g)$$
 for all $i \leq k$ by a similar argument as in lemma \ref{Lem:produktallg}. 
Then it is clear that $\Der_{S,p}(\log D_1)$ can be generated by some $\delta_1, \ldots, \delta_m$ and $\D_{y_{k+1}}, \ldots, \D_{y_n}$, where $m \geq k$ and the coefficients of the $\delta_i$ only
depend on $x_1, \ldots, x_k$ and lie in the maximal ideal $\mf{m} \subseteq \calo$. \\
Using $\D_{x_i} \in {}_\calo \langle \varepsilon_1, \ldots, \varepsilon_l, \delta_1, \ldots, \delta_m, \D_{y_{k+1}}, \ldots, \D_{y_n} \rangle$ and computing $\D_{x_i}h$ we obtain (similar to the proof
of lemma \ref{Lem:produktallg}) that
$$(\D_{x_1}h, \ldots, \D_{x_k}h) \subseteq (h, \D_{y_{k+1}}h, \ldots, \D_{y_n}h).$$
An application of the triviality lemma shows that $D_2$ can be chosen independently of the variables $x_1, \ldots, x_k$, and thus that $D$ is splayed.
\end{proof}

\begin{question}
In view of Proposition \ref{Thm:splayedder} and the duality of $\Der_{S,p}(\log D)$ and $\Omega^1_{S,p}(\log D)$ we ask: can one express the splayedness of a divisor $(D_1 \cup D_2,p)$ in
terms of $\Omega^1_{S,p}(\log D_1)$, $\Omega^1_{S,p}(\log D_2)$ and $\Omega_{S,p}^1$?
\end{question}

\section{Applications}

\subsection{Free divisors and normal crossings} \label{Sub:freediv}

In some sense free divisors are a generalization of normal crossing divisors: their modules of tangent vector fields are free. However, irreducible (non-smooth) free divisors are highly singular,
that is, they are non-normal and their singular loci are Cohen--Macaulay of codimension 1 in the divisor, see \cite{Aleksandrov90,Simis06}. Here we consider the relationship between free and splayed
divisors and in particular
between splayed and normal crossing divisors. First we show that a splayed divisor is free if and only if its splayed components are free (Proposition \ref{Prop:produktfreiallg}). Then we turn to normal
crossing divisors, in particular to the problem of characterizing normal crossing divisors by their Jacobian ideal. This problem was stated by H.~Hauser and considered in \cite{Faber11}, where a
general answer was found. Here we use the ideal theoretic characterizations of splayed divisors in order to show that a divisor consisting of smooth irreducible components has normal crossings at a
point $p$ if and only if it is free at $p$ and its Jacobian ideal is radical (Corollary \ref{Cor:smoothcomp}). \\

First let us describe the structure of $\Der_{S,p}(\log D)$ for splayed divisors in a different way: let $S,T$ be complex manifolds of dimensions $n,m$ and suppose that $(S\times
T,0)\cong (\C^{n+m},0)$, with complex coordinates $(x,y)=(x_1, \ldots, x_n,$ $y_1,
\ldots, y_m)$ at the origin. Let $(D^x_1,0)$ be a divisor in $(S,0)$,
which is defined by a reduced $g'(x) \in \mc{O}_{S,0}\cong \C\{x_1, \ldots,
x_n\}$ and which has a logarithmic derivation module over $\C\{x\}$ denoted by
$\Der_{S,0}(\log D^x_1)$. Then we may consider the cylinder over $D^x_1$ in the
$T$-direction in $(S\times T,0)$, namely the hypersurface $D_1$ defined by $g(x,y)=g(x,0):=g'(x) \in \C\{x,y\}$. It is easy to see that  
$$\Der_{S\times T,0}(\log D_1)=(\Der_{S,0}(\log
D^x_1)\otimes_{\C\{x\}}\C\{x,y\})\oplus (\Der_{T,0}
\otimes_{\C\{y\}}\C\{x,y\}).$$ 

Similarly define $D_2^y$ and $D_2$ with equations $h'(y)=h(x,y)$ and also $\Der_{S\times T,0}(\log D_2)$.
Thus we define
the (splayed) divisor $D=D_1 \cup D_2$ in $S \times T$ that is given at $0$ by
the equation $gh=0$. Since $g$ and $h$ have separated variables, there is a natural
splitting of $\Der_{S \times T,0}(\log D)$:
{\small \begin{equation}  \label{Equ:splayedvectorfields}
 \Der_{S\times T,0}(\log D)\!=(\Der_{S,0}(\log D^x_1)\otimes_{\C\{x\}}\C\{x,y\})
\oplus (\Der_{T,0}(\log D^y_2)\otimes_{\C\{y\}}\C\{x,y\}).
\end{equation}
} \normalsize

\begin{proposition} \label{Prop:produktfreiallg} 
 Let $(D,p)=(D_1,p) \cup (D_2,p)$ be a splayed divisor in $S \cong \C^{n+m}$ defined as above. 
The divisor $D=\{g(x)h(y)=0\}$ is
free if and only if $D_1=\{g(x)=0\}$ and $D_2=\{h(y)=0\}$ are both free.
\end{proposition}

\begin{proof}
If both $D_1$ and $D_2$ are free then there exist bases of
$\Der_{S\times T,p}(\log D_1)$ and $\Der_{S\times T,p}(\log D_2)$ of the form
$$\delta_1=\sum_{i=1}^n a_{1i} \D_{x_i}, \ldots, \delta_n=\sum_{i=1}^n a_{ni}
\D_{x_i}, \delta_{n+1}=\D_{y_1}, \ldots, \delta_{n+m}=\D_{y_m}$$ 
and
$$\varepsilon_1=\D_{x_1}, \ldots, \varepsilon_n=\D_{x_n}, \varepsilon_{n+1}=\sum_{i=1}^mb_{n+1,i}\D_ {y_i}, \ldots,
\varepsilon_{n+m}=\sum_{i=1}^mb_{n+m,i}\D_ {y_i}.$$ 
It is easy to see that any $\delta_i$ for $1 \leq i \leq n$ and any $\varepsilon_j$ for $n+1 \leq j \leq n+m$ is also an element of $\Der_{S\times T,p}(\log
D)$. 
By Saito's criterion (Thm.~\ref{Thm:Saito}) it follows that $\delta_1, \ldots, \delta_n,
\varepsilon_{n+1},
\ldots, \varepsilon_{n+m}$ form a basis of $\Der_{S \times T,p}(\log D)$. Conversely, suppose that $\Der_{S\times T,p}(\log D)$ is free. 
Since $\Der_{S\times T,p}(\log D)$ is free, it follows by (\ref{Equ:splayedvectorfields})
that $\Der_{S,0}(\log D^x_1)\otimes_{\C\{x\}}\C\{x,y\}$ and 
$\Der_{T,0}(\log D^y_2)\otimes_{\C\{y\}}\C\{x,y\}$ are projective
$\mc{O}_{S\times T,p}$-mod\-ules. Since the notion of projective and free module over regular
local rings coincide, 
these two
modules are even free.
\end{proof}

\subsubsection{Radical Jacobian ideal - normal crossings}

In \cite{Faber11} the following problem (proposed by H.~Hauser) was considered: suppose that $D \subseteq S$, $\dim S=n$, is a divisor that is locally at a point $p \in S$ given by $\{h=0\}$. Can we
determine if $D$ has normal crossings at $p$ by the knowledge of its Jacobian ideal $J_h$? \\
It was shown that $D$ has normal crossings in case $D$ is free, the Jacobian ideal is radical and the normalization of $D$ is smooth (see
\cite[Thm.~2.4]{Faber12}).   \\

Here we use splayed divisors to show this theorem for the case where $(D,p)=\bigcup_{i=1}^m(D_i,p)$ is a union of \emph{smooth} irreducible components $D_i$.  Note that in this special case the
hypothesis on the normalization is not needed. Moreover, the problem can be reduced to the case where $(D,p)$ is irreducible (Proposition \ref{Prop:produktfreireduktion}).\\

\begin{rem}  \label{Rmk:BrianconSkoda}
One can show, using the theorem of Brian{\c{c}}on--Skoda that $J_h$ radical implies that $h$ is already contained in $J_h$, see  \cite{Faber12}.
\end{rem}

\begin{lem} \label{Lem:produktfreiradikal} 
 Let $D=\{gh=0\}$ be splayed at a point $p=(x_1, \ldots, x_k,$ $y_{k+1}, \ldots, y_n)$ in $S \cong \C^{n}$ and denote by $D_1=\{g=g(x,0)=0\}$ and $D_2=\{h=h(0,y)=0\}$ its splayed components.  
The
Jacobian ideal of $D$, denoted by $J_{gh}$ is radical if and only if both $J_h$ and
$J_g$ are also radical. 
\end{lem}

\begin{proof}
Suppose that $J_g$ and $J_h$ are radical. Similarly to the proof of Proposition \ref{Thm:splayedgeom} one concludes that
the ideals  $J_{gh}$ and $(g,h)\cap J_g \cap J_h$ are equal. We compute the radical 
$$\sqrt{J_{gh}}=\sqrt{(g,h)\cap J_g \cap J_h}=\sqrt{(g,h)} \cap \sqrt{J_g} \cap
\sqrt{J_h}=(g,h)\cap J_g \cap J_h=J_{gh},$$ 
where the third equality holds because of our assumptions. 
Conversely, suppose
that $J_{gh}=\sqrt{J_{gh}}$. Then $gh$ is an element of $J_{gh}$ (cf.~Remark  \ref{Rmk:BrianconSkoda}).
Localization of $\C\{x,y\}$ in $g$ yields $(J_{gh})_g=((h)+J_h)_g$, which is radical, since
$J_{gh}$ is radical. Note that for an ideal $I \subseteq \C\{x,y\}$, we denote by $I_g$ the localization of $I$ in $g$ (cf.~\cite{Matsumura86}). Using that $((h)+J_h)_g$ is an intersection of prime
ideals in the localization, a direct computation shows that any element of $\sqrt{((h)+J_h)}$ is already contained in $((h)+J_h)$.
Similarly one proves
$((g)+J_g)=\sqrt{((g)+J_g)}$.  
\end{proof}

\begin{lem} \label{Lem:produktradikal}
Let $D\subseteq S$ be a divisor given at $p \in S$ by $\{ gh=0\}$ with $gh \in \mc{O}_{S,p}$ reduced and suppose that $J_{gh}$ is radical. Then 
$$J_{gh}=gJ_h+hJ_g.$$
In particular, $D=D_1 \cup D_2$ defined by $D_1=\{g=0\}$ and $D_2=\{h=0\}$ is splayed at $p$.
\end{lem}

\begin{proof}
Consider the element $g\D_{x_i}(gh)$, which is contained in $J_{gh}$ for any $i=1, \ldots, n$. Since $J_{gh}$ is radical, $gh$ is contained in $J_{gh}$, which forces $g^2(\D_{x_i}h)$ to be in
$J_{gh}$. Using that $J_{gh}$ is an intersection of prime ideals $\mf{p}_j$, $j=1, \ldots, k$, one obtains that either $g$ or $\D_{x_i}h$ is contained in any $\mf{p_j}$. Hence $g(\D_{x_i}h) \in
J_{gh}$. This yields
$$J_{gh}=(g \D_{x_1}(h), \ldots, g \D_{x_n}(h), h \D_{x_1}(g), \ldots, h \D_{x_n}(g))=g
J_h + h J_g.$$
By Proposition \ref{Thm:Leibnizprod} $(D,p)$ is splayed.
\end{proof}

\begin{exa}
Splayed divisors need not have radical Jacobian ideals, as the following example shows. Let $D$ be the divisor in $(\C^3,0)$ with coordinates $(x,y,z)$ at
$0$, that is defined by $gh=x(y^2+z^3)$. Then clearly $(D,0)$ is splayed. The
Jacobian ideal is $J_{gh}=(y^2+z^3,xy,xz^2)=(y,z^2) \cap (x,y^2+z^3)$, which is
not radical. Note that $D$ is a free divisor.
\end{exa}

\begin{proposition} \label{Prop:produktfreireduktion}
Let $D=D_1 \cup D_2$ be a divisor in an $n$ dimensional complex manifold $S$ and let $D$, $D_1$ and $D_2$ at a point $p \in S$ be defined by the equations  $gh$,  $g$ and $h$, respectively. Suppose
that $J_{gh}$ is radical. Then $D$ is splayed and $J_h$ and
$J_g$ are also radical. If moreover $D$ is free at $p$ then also $D_1$ and $D_2$
are free at $p$.
\end{proposition}

\begin{proof}
This follows from proposition \ref{Thm:Leibnizprod}, proposition \ref{Prop:produktfreiallg} and lemma \ref{Lem:produktradikal}.
\end{proof}

\begin{cor}  \label{Cor:smoothcomp}
Let $(S,D)$ be a complex manifold, $\dim S=n$, together with a divisor $D \subseteq S$ and suppose that locally at a point $p \in S$ the divisor $(D,p)$ has a  decomposition into irreducible
components $\bigcup_{i=1}^m
(D_i,p)$ such that each $(D_i,p)$ is smooth. Let the corresponding equation of $D$ at $p$ be $h=h_1 \cdots h_m$.
If $D$ is free at $p$ and $J_h=\sqrt{J_h}$ then $D$ has normal crossings at $p$.
\end{cor}

\begin{proof}
We use induction on $n$. If $n=2$, then an explicit computation or an application of the theorem of Mather--Yau (\cite{MY}) shows the assertion. 
Now suppose the assertion is true for divisors in manifolds of dimension $n-1$. For a smooth component $D_1$ of $D$, one can
find local coordinates $(x_1, \ldots, x_n)$ such that $D_1=\{x_1
=0\}$. Proposition \ref{Prop:produktfreireduktion} shows that the divisor $(D \setminus D_1):=h_2 \cdots
h_m$ is also free and has a radical Jacobian ideal. Moreover, $D $ is locally 
splayed, that is, $D \setminus D_1$ is locally isomorphic to some divisor depending
only on the last $n-1$ coordinates. Thus by induction hypothesis $D \setminus
D_1$ is isomorphic to a normal crossings divisor $y_2 \cdots y_m=0$, where the $y_i$ are the result of a coordinate transformation of $(x_1, \ldots, x_n)$ such that $x_1=y_1$. 
This implies that $m \leq n-1$. Hence $D$ is isomorphic to the normal
crossings divisor $x_1 y_2 \cdots y_m$. 
\end{proof}

\subsection{Hilbert--Samuel polynomials} \label{Sub:HSP}

Splayed divisors are particularly interesting for
computational reasons. We start here a study of properties of splayed divisors by considering their Hilbert--Samuel polynomials. We find that multiplicities behave the same for splayed as for
non--splayed divisors but that the
Hilbert--Samuel polynomials for splayed divisors are additive, which means the
following: let $(D,p)=(D_1,p) \cup (D_2,p)$ be a splayed divisor at a point $p$ in a complex manifold. Then from the exact sequence 
$$ 0 \rightarrow \mc{O}_{D,p} \rightarrow \mc{O}_{D_1,p} \oplus \mc{O}_{D_2,p}
\rightarrow \mc{O}_{D_1 \cap D_2,p} \ra 0$$
follows
$$\chi_{D,p}+\chi_{D_1 \cap D_2,p}=\chi_{D_1,p} + \chi_{D_2,p},$$ 
where $\chi_{D,p}$ denotes the Hilbert--Samuel polynomial of $D$ at $p$. Since one can compute Hilbert--Samuel polynomials of divisors in an easy way (\cite[Lemma 4.2.20]{dJP}), this yields a method
to compute $\chi_{D_1 \cap D_2,p}$. \\

The Hilbert--Samuel function and the Hilbert--Samuel polynomial 
do not only depend on the module one is considering but also on a chosen filtration.
However, the degree and leading coefficient of the Hilbert--Samuel polynomial are
independent of the filtration. 
We use notation from \cite{Matsumura86} and
\cite{dJP}. \\

Let $R$ be a noetherian local ring with maximal ideal $\mf{m}$. Let $I \subseteq R$ be an ideal and let $M$ be a module over $R$. A set $\{ M_n\}_{n \geq 0}$ of submodules of $M$ is called an
$I$-\emph{filtration} of $M$ if $M=M_0 \supset M_1 \supset M_2 \supset \ldots$ and $IM_n \subseteq M_{n+1}$ for all $n \geq 0$. \\
In the following we will always suppose that $M$ is a finitely generated $R$ module. Let $\mf{q}$ be an $\mf{m}$-primary ideal of $R$ and let $\{M_i\}$ be a $\mf{q}$-filtration. Then the
\emph{Hilbert--Samuel function of the filtration $\{M_i\}$} is  
$$HS_{\{M_i\}}: \N \rightarrow \N, d \mapsto \mathrm{length}_{R/ \mf{m}}M/ M_d.$$
If $M_i=\mf{m}^i M$ then we simply write $HS_M$. One can show that there exists a polynomial $\chi_{\{
M_i \}}$ with rational coefficients such that $HS_{\{M_i\}}(d)=\chi_{\{M_i\}}(d)$
for $d$ sufficiently large. We call $\chi_{M}^\mf{q}:=\chi_{\{\mf{q}^n M\}_{n
\geq 0}}$ the \emph{Hilbert--Samuel polynomial of $M$ with respect to $\mf{q}$}.
The degree $d$ of $\chi_{M}^\mf{q}(k)=\sum_{i=0}^d a_i k^i$ only depends on $M$
and not on $\mf{q}$. \\

When considering Hilbert--Samuel polynomials of modules over a local ring
$\mc{O}=\C\{x_1, \ldots, x_n\}$ with respect to the maximal ideal $\mf{m}$, one can use
 standard bases to simplify computations. 
For definitions and notation about standard bases we refer to \cite{dJP}. Here we only need the following facts: \\

{\bf Fact (i)} Take the degree lexicographical ordering $<$. Then for an ideal $I \subseteq \calo$ one has
$$HS_{\calo/I}(k)=HS_{\calo/L(I)}(k),$$ 
where $L(I)$ denotes the leading ideal of $I$. In particular,
$\calo/I$ and $\calo/L(I)$ have the same Hilbert--Samuel polynomial with
respect to $\mf{m}$. \\
{\bf Fact (ii)} Let $f, g \in \calo$ and assume that $L(f)$ and $L(g)$ are coprime. Then $L((f,g))=(L(f),L(g))$, that is, $f,g$ are a standard basis of the ideal
$(f,g)$. \\

In contrast to the graded case, the Hilbert--Samuel polynomial is not additive on exact sequences, one has a certain error polynomial, whose degree can be determined with a theorem of Flenner and
Vogel, see \cite{FlennerVogel93}:

\begin{thm}[Flenner--Vogel] \label{Thm:FlennerVogel}
Let $(R, \mf{m})$ be a noetherian local ring, $0 \rightarrow M' \rightarrow M \rightarrow M'' \rightarrow 0$ an exact sequence of  finitely generated $R$-modules and $\mf{q}$ an $\mf{m}$-primary
ideal.  Denote further
$$\mathrm{Gr}_\mf{q}(N)=\bigoplus_{i=0}^\infty \mf{q}^iN/\mf{q}^{i+1}N$$
 the associated graded module of a finite $R$-module $N$. Then the following holds: \\
(a) $\mathrm{supp} \ker (\mathrm{Gr}_\mf{q}(M') \rightarrow \mathrm{Gr}_\mf{q}(M))=\mathrm{supp} \ker (\mathrm{Gr}_\mf{q}(M)/\mathrm{Gr}_\mf{q}(M') \rightarrow \mathrm{Gr}_\mf{q}(M''))$. \\
(b) Denote by $d$ the dimension of these supports. Then for all there is a polynomial $S$ of degree $d-1$ such that for $n \gg
0$ we have
$$ S(n):=\chi_{M'}^\mf{q}(n) + \chi_{M''}^\mf{q}(n)-\chi_{M}^\mf{q}(n).$$ 
In particular, if
$d=0$, then $S=0$.
\end{thm}

Before turning to splayed divisors, consider the additivity problem of the Hilbert--Samuel polynomial for arbitrary finitely generated modules over a local ring $(R, \mf{m})$.
If 
$$0 \rightarrow N \rightarrow M \rightarrow M/N \rightarrow 0$$
is an exact sequence of finitely generated $R$-modules and $\mf{q}$ an $\mf{m}$-primary ideal, then 
\begin{equation} \label{Equ:HSPerror} 0 \rightarrow N/(\mf{q}^n M \cap N) \rightarrow M/\mf{q}^nM \rightarrow (M/N)/\mf{q}^n(M/N) \rightarrow 0 
\end{equation}

is an exact sequence, which implies 
$$\chi_{M}^\mf{q}=\chi_{M/N}^\mf{q}+\chi_{\{\mf{q}^nM \cap N\}}.$$
Here $\chi_{\{\mf{q}^nM \cap N \}}$ denotes the Hilbert--Samuel polynomial with respect to the filtration $N/(\mf{q}^nM \cap N)$ and in general one has $\chi_{N}^\mf{q} \neq \chi_{\{\mf{q}^nM \cap N\}}$. 
However, for split exact sequences the Hilbert--Samuel polynomial is always additive:

\begin{lem} \label{Lem:hilbertdirektesumme}
Let $(R, \mf{m})$ be a local ring and let $M, N$ be finitely generated $R$-modules. Consider the exact sequence 
$$0 \rightarrow N \rightarrow N \oplus M \rightarrow M \rightarrow 0.$$
Then $\chi_{N \oplus M}^\mf{q}=\chi_{M}^\mf{q}+ \chi_{N}^\mf{q}$ for any $\mf{m}$-primary ideal $\mf{q}$.
\end{lem}

\begin{proof}
A direct computation shows that $N \cap \mf{q}^n (N\oplus M)=\mf{q}^n N$ for any
$n$. Using the exact sequence (\ref{Equ:HSPerror}) the assertion follows.
\end{proof}

Let now
$D=D_1 \cup D_2\subseteq \C^{n}$ be a not necessarily splayed divisor that is locally at a point $p=(x)$ defined by $h_1(x) h_2(x) \in \mc{O}=\C\{x\}$, with components $(D_1,p)=\{
h_1(x)=0\}$ and $D_2=\{h_2(x)=0\}$. 
The multiplicities of $D_i$ at $p$ are denoted by $m_p(D_i):=\mf{e}(\mc{O}_{S,p}/(h_i), \mf{m})$.
The Hilbert--Samuel polynomial of $D_i$ at $p$ is denoted by
$\chi_{D_i,p}:=\chi_{\mc{O}/(h_i)}^\mf{m}$, and similarly the multiplicity and Hilbert--Samuel polynomial for $D$. 

\begin{rem}
In order to compute the Hilbert--Samuel polynomial of $\mc{O}/I$ for any ideal
$I \subseteq \mc{O}$ we can consider $\mc{O}/I$ either as ring or as an
$\mc{O}$-module. This does not make a difference for the Hilbert--Samuel
functions, since they only depend on the graded structure of $\mc{O}/I$.
\end{rem}

It is well-known that multiplicities of divisors are additive, more precisely one can show that for $D$ and $D_i$ defined as above, $m_p(D)= m_p(D_1) + m_p(D_2)$ (see e.g. \cite{dJP}). For
splayed divisors the additivity also holds for Hilbert--Samuel polynomials:

\begin{proposition} \label{Prop:HSPadditiv}
Let $(D,p)=(D_1,p) \cup (D_2,p)$ be splayed at $p \in S$, where $(S,p) \cong (\C^n,0)$. Then the Hilbert--Samuel polynomials of the components $D_1$ and $D_2$ are additive, that is,
$$\chi_{D,p}(t)+\chi_{D_1 \cap D_2,p}(t)=\chi_{D_1,p}(t)+\chi_{D_2,p}(t).$$
\end{proposition}

\begin{proof}
Denote by $\mc{O}:=\C\{x,y\}=$ $\C\{x_1, \ldots, x_k,$ $y_{k+1},\ldots, y_n\}$ the coordinate ring of $(\C^n,0)$. 
There is an exact sequence (see e.g. \cite{GreuelPfister08})
\begin{equation} \label{Diag:multiplicitiesexakt}
0 \rightarrow \mc{O} /( h_1 h_2) \rightarrow \mc{O}/ (h_1) \oplus \mc{O}/
(h_2) \rightarrow \mc{O}/(h_1,h_2) \rightarrow 0.
\end{equation} 
With facts (i) and (ii) from above the question can be reduced to leading
ideals because (\ref{Diag:multiplicitiesexakt}) remains exact if we just consider the leading ideals. 
The divisor $D$ is splayed, so we can assume that it is defined by $g(x)h(y)$. Choosing any valid monomial
ordering one finds
 $L(g)=x^\alpha$, $L(h)=y^\beta$, $L((gh))=x^\alpha y^\beta$, and by fact (ii) from above follows $L((g,h))=(x^\alpha, y^\beta)$. From 
 Lemma \ref{Lem:hilbertdirektesumme} it follows that
$$\chi_{\mc{O}/(x^\alpha)}^\mf{m}+\chi_{\mc{O}/(y^\beta)}^\mf{m}=\chi_{\mc{O}/
(x^\alpha) \oplus \mc{O}/
(y^\beta)}^\mf{m}.$$
Thus it remains to prove that the Hilbert--Samuel polynomials with respect to $\mf{m}$ of
the exact sequence
$$0 \rightarrow \mc{O}/(x^\alpha \cdot y^\beta) \rightarrow \mc{O}/(x^\alpha)
\oplus \mc{O}/(y^\beta) \rightarrow \mc{O}/(x^\alpha, y^\beta)\rightarrow 0$$
are additive. In order to apply the theorem of Flenner--Vogel we 
show that the map 
$$ \xymatrix{ \mathrm{Gr}_\mf{m}(\mc{O}/(x^\alpha y^\beta)) \ar[r]^-\varphi & 
\mathrm{Gr}_\mf{m}(\mc{O}/(x^\alpha) \oplus \mc{O}/(y^\beta))}$$
is injective. The map $\varphi$ clearly preserves the degree, so
it is
enough to show the assertion for a homogeneous element of degree $d$.
Therefore take
some $\overline{a} \in \mf{m}^d (\mc{O}/(x^\alpha y^\beta))/ \mf{m}^{d+1}
(\mc{O}/(x^\alpha y^\beta))$.  
Consider $\overline{a}$ as an element in $\mc{O}$: from Grauert's division
theorem follows that $\overline{a}$ can be written as $r + \alpha_1 x^\alpha + \alpha_2
y^\beta$, where $r$ is the unique remainder from the division through
$x^\alpha$ and $y^\beta$ and $\alpha_1 \in \mc{O}$ is not divisible by $y^\beta$
and
$\alpha_2$ is not divisible by $x^\alpha$. 
Suppose that $\varphi(\overline{a})=(0,0)$. Then 
 write $\varphi(\overline{a})=(r + \alpha_2
y^\beta, r + \alpha_1 x^\alpha)$ in $\mc{O}/(x^\alpha) \oplus \mc{O}/(y^\beta)$. In $\mc{O}$ this reads as $r +
\alpha_2 y^\beta = c x^\alpha$ and $r + \alpha_1 x^\alpha = c' y^\beta$ for
some $c,c' \in \mc{O}$ with the right order. It follows that $r \in (x^\alpha , y^\beta)$. But $r$ is the unique
remainder from the division through the standard basis $(x^\alpha, y^\beta)$, so
$r=0$ in $\mc{O}$.
Since $x^\alpha, y^\beta$ are clearly a regular sequence in $\mc{O}$, their
syzygies are trivial and from the conditions on $\alpha_1, \alpha_2$ it follows
that $\alpha_1=\alpha_2=0$. This implies the injectivity of $\varphi$. 
Hence by the theorem of Flenner--Vogel, the remainder polynomial
$S$ is zero and the assertion of the proposition follows. 
\end{proof}

In general the Hilbert--Samuel polynomial of a divisor $(D,p)=(D_1,p)\cup
(D_2,p)$ is not additive, as is seen in the following example.

\begin{exa}
By Lemma 4.2.20 of \cite{dJP} one can explicitly compute the Hilbert--Samuel
polynomial of $\mc{O}/(f)$, where $\mc{O}=\C \{x_1, \ldots, x_n\}$ and
$\mathrm{ord}(f)=m$, namely
\begin{equation} \chi_{\mc{O}/(f)}^\mf{m}(d)=\sum_{j=1}^m {n+d-j-1 \choose n-1}. \label{Eq:HShyper}   \end{equation}
Consider now $(\C^2,0)$ with coordinate ring $\mc{O}=\C\{x,y\}$ and with
$h_1=x^2-y$ and $h_2=y$. Then the germ of the divisor $(D,0)=(D_1,0)\cup (D_2,0)$ that is locally
given by $\{y(x^2-y)=0\}$ with $(D_1,0)=\{x^2-y=0\}$
and $(D_2,0)=\{y=0\}$ is not splayed. The intersection $(D_1 \cap D_2,0)$ is locally given by
the ideal $(x^2,y)$ and coordinate ring $\mc{O}/(x^2,y)=\C\{x\}/(x^2)$. By 
formula (\ref{Eq:HShyper}) we can compute the Hilbert--Samuel polynomials of $D,D_1,D_2$ and
$D_1 \cap D_2$ and obtain $\chi_{D,p}(t)=2t-1$, which is clearly not equal to
$\chi_{D_1,p}(t) +\chi_{D_2,p}(t)- \chi_{D_1 \cap D_2,p}(t)= t + t -2$.
\end{exa}

One might ask if the additivity of the Hilbert--Samuel polynomials
characterizes splayed divisors. However, here is a counterexample to this assertion:

\begin{exa}
Consider $D
\subseteq \C^3$ locally defined by $gh=(x^2-y^3)(y^2-x^2z)$. Then $(D,p)$ is
the union of the cylinder over a cusp $(D_1,p)$ and of the Whitney Umbrella
$(D_2,p)$. Clearly $(D,p)$ is not splayed (use for example the
Leibniz--property). However, $L(g)=x^2$ and $L(h)=y^2$, so the leading
monomials of $g$ and $h$ are coprime and one can repeat the argument in the
proof of the preceding proposition to find that 
$$ \chi_{D,p}+ \chi_{D_1 \cap D_2,p}=\chi_{D_1,p} + \chi_{D_2,p}.$$
\end{exa}

\subsection*{Acknowledgments} I thank my advisor H. Hauser for getting me interested in this topic and helping me with comments and suggestions. I thank D.~Mond and L.~Narv{\'a}ez Macarro for several
discussions and comments on this work, and also P.~Aluffi, M.~Granger, H.~Kawanoue, S.~Perlega, D.~B.~Westra and the  referee for helpful comments on an earlier version of this paper. In particular I thank O.~Villamayor and his
research group at
the Universidad Aut{\'o}noma
de Madrid for their hospitality and for many discussions in which the basis for the results presented here was laid.

\bibliographystyle{alpha}
\bibliography{biblioNC}

\def\cprime{$'$}
\begin{thebibliography}{BEvB09}

\bibitem[AF12]{AluffiFaber12}
P.~Aluffi and E.~Faber.
\newblock {Splayed divisors and their Chern classes}.
\newblock 2012.
\newblock {\tt arXiv:1207.4202v2 [math.AG]}.

\bibitem[Ale86]{Aleksandrov86}
A.~G. Aleksandrov.
\newblock Euler-homogeneous singularities and logarithmic differential forms.
\newblock {\em Ann. Global Anal. Geom.}, 4(2):225--242, 1986.

\bibitem[Ale90]{Aleksandrov90}
A.~G. Aleksandrov.
\newblock {Nonisolated Saito singularities}.
\newblock {\em Math. USSR Sbornik}, 65(2):561--574, 1990.

\bibitem[BEvB09]{Buchweitz06}
R.-O. Buchweitz, W.~Ebeling, and H.C.~Graf von Bothmer.
\newblock Low-dimensional singularities with free divisors as discriminants.
\newblock {\em J. Algebraic Geom.}, 18(2):371--406, 2009.

\bibitem[BM06]{BuchweitzMond}
R.-O. Buchweitz and D.~Mond.
\newblock Linear free divisors and quiver representations.
\newblock In {\em Singularities and computer algebra}, volume 324 of {\em
  London Math. Soc. Lecture Note Ser.}, pages 41--77. Cambridge Univ. Press,
  Cambridge, 2006.

\bibitem[Bod04]{Bodnar04}
G.~Bodn{\'a}r.
\newblock Algorithmic tests for the normal crossing property.
\newblock In {\em Automated deduction in geometry}, volume 2930 of {\em Lecture
  Notes in Comput. Sci.}, pages 1--20. Springer, Berlin, 2004.

\bibitem[Bud12]{Budur10}
N.~Budur.
\newblock {Singularity invariants related to Milnor fibers: survey}.
\newblock In {\em Recent Trends in Zeta Functions in Algebra and Geometry},
  volume 566 of {\em Contemp. Math.}, pages 161--187, Providence, RI, 2012.
  Amer. Math. Soc.

\bibitem[CNM96]{CNM96}
F.~Castro\phantom{l}Jim{\'e}nez, L.~Narv{\'a}ez\phantom{l}Macarro, and D.~Mond.
\newblock Cohomology of the complement of a free divisor.
\newblock {\em Trans. Amer. Math. Soc.}, 348(8):3037--3049, 1996.

\bibitem[Dam96]{Damon96}
J.~Damon.
\newblock Higher multiplicities and almost free divisors and complete
  intersections.
\newblock {\em Mem. Amer. Math. Soc.}, 123(589):x+113, 1996.

\bibitem[dJP00]{dJP}
T.~de~Jong and G.~Pfister.
\newblock {\em Local analytic geometry}.
\newblock Advanced Lectures in Mathematics. Vieweg, Braunschweig -- Wiesbaden,
  2000.

\bibitem[Fab11]{Faber11}
E.~Faber.
\newblock Normal crossings in local analytic geometry.
\newblock Phd thesis, Universit\"at Wien, 2011.

\bibitem[Fab12]{Faber12}
E.~Faber.
\newblock {Characterizing normal crossing hypersurfaces}.
\newblock 2012.
\newblock {\tt arXiv:1202.6276v1 [math.AG]}.

\bibitem[FH10]{FaberHauser10}
E.~Faber and H.~Hauser.
\newblock Today's menu: geometry and resolution of singular algebraic surfaces.
\newblock {\em Bull. Amer. Math. Soc.}, 47(3):373--417, 2010.

\bibitem[FV93]{FlennerVogel93}
H.~Flenner and W.~Vogel.
\newblock On multiplicities of local rings.
\newblock {\em Manuscripta Math.}, 78(1):85--97, 1993.

\bibitem[GH85]{GH}
T.~Gaffney and H.~Hauser.
\newblock Characterizing singularities of varieties and of mappings.
\newblock {\em Invent. Math.}, 81(3):427--447, 1985.

\bibitem[GP08]{GreuelPfister08}
G.-M. Greuel and G.~Pfister.
\newblock {\em A Singular introduction to commutative algebra}.
\newblock Springer-Verlag, Berlin, 2008.
\newblock Second, extended edition. With contributions by Olaf Bachmann,
  Christoph Lossen and Hans Sch{\"o}nemann. With 1 CD-ROM (Windows, Macintosh
  and UNIX).

\bibitem[GS09]{GrangerSchulze09}
M.~Granger and M.~Schulze.
\newblock Initial logarithmic {L}ie algebras of hypersurface singularities.
\newblock {\em J. Lie Theory}, 19(2):209--221, 2009.

\bibitem[HM93]{HM93}
H.~Hauser and G.~M{\" u}ller.
\newblock {Affine varieties and Lie algebras of vector fields}.
\newblock {\em Manuscripta Math.}, 80:309--337, 1993.

\bibitem[Li09]{LiLi}
Li~Li.
\newblock Wonderful compactification of an arrangement of subvarieties.
\newblock {\em Michigan Math. J.}, 58(2):535--563, 2009.

\bibitem[Mat86]{Matsumura86}
H.~Matsumura.
\newblock {\em Commutative ring theory}, volume~8 of {\em Cambridge Studies in
  Advanced Mathematics}.
\newblock Cambridge University Press, Cambridge, 1986.

\bibitem[MS10]{MondSchulze10}
D.~Mond and M.~Schulze.
\newblock {Adjoint divisors and free divisors}.
\newblock 2010.
\newblock {\tt arXiv:1001.1095v2 [math.AG]}.

\bibitem[MY82]{MY}
J.~N. Mather and S.~S.-T. Yau.
\newblock {Classification of Isolated Hypersurface Singularities by their
  Moduli Algebras}.
\newblock {\em Invent. Math.}, 69:243--251, 1982.

\bibitem[OT92]{OrlikTerao92}
P.~Orlik and H.~Terao.
\newblock {\em Arrangements of hyperplanes}, volume 300 of {\em Grundlehren der
  Mathematischen Wissenschaften}.
\newblock Springer-Verlag, Berlin, 1992.

\bibitem[Sai80]{Saito80}
K.~Saito.
\newblock Theory of logarithmic differential forms and logarithmic vector
  fields.
\newblock {\em J. Fac. Sci. Univ. Tokyo}, 27(2):265--291, 1980.

\bibitem[Sai81]{Saito81}
K.~Saito.
\newblock Primitive forms for a universal unfolding of a function with an
  isolated critical point.
\newblock {\em J. Fac. Sci. Univ. Tokyo Sect. IA Math.}, 28(3):775--792, 1981.

\bibitem[Sek08]{Sekiguchi08}
J.~Sekiguchi.
\newblock {Three Dimensional Saito Free Divisors and Singular Curves}.
\newblock {\em J. Sib. Fed. Univ. Math. Phys.}, 1:33--41, 2008.

\bibitem[Sim06]{Simis06}
A.~Simis.
\newblock Differential idealizers and algebraic free divisors.
\newblock In {\em Commutative algebra}, volume 244 of {\em Lect. Notes Pure
  Appl. Math.}, pages 211--226. 2006.

\end{thebibliography}



\end{document}